\documentclass[10pt]{article}
\usepackage{hyperref}
\usepackage[english]{babel}
\usepackage{amssymb,amsmath, amsthm}
\usepackage{graphicx}
\usepackage{graphics}
\usepackage{psfrag}

\usepackage[all]{xy}
\usepackage[applemac]{inputenc}

\newtheorem{axiom}{Assumption}[section]

\newtheorem{lemma}{Lemma}[section]

\newtheorem{prop}{Proposition}[section] 
\newtheorem{defn}{Definition}[section]
\newtheorem{apen}{Appendix}[section]

\theoremstyle{definition}

\newtheorem{exam}{Example}[section]

 \author{Mário M. Gra\c{c}a \\
 \\
\small  Department of Mathematics, Instituto Superior Técnico,\\
\small Technical University of Lisbon\\
\small  Av. Rovisco Pais, 1049\--001, Lisboa, Portugal\\
\small  \url{mgraca@math.ist.utl.pt}}

\begin{document}
 
\title{A SIMPLE DERIVATION OF NEWTON\--COTES FORMULAS WITH REALISTIC ERRORS}
\maketitle

\noindent

\begin{abstract}

\noindent

\noindent
In order to approximate the integral $I(f)=\int_a^b f(x) dx$, where $f$ is a sufficiently smooth function, models for  quadrature rules are developed using a given {\it panel} of  $n\,\,(n\geq 2)$ equally spaced  points.   These models arise from the undetermined coefficients method, using a Newton's basis for polynomials. Although part of the final product is algebraically  equivalent to the well known closed Newton\--Cotes rules, the algorithms  obtained are not the classical ones. 

\noindent
  In the basic model  the most simple quadrature rule $Q_n$ is adopted (the so\--called left rectangle rule) and a correction  $\tilde E_n$ is constructed, so that the final rule  $S_n=Q_n+\tilde E_n$ is interpolatory.  The correction $\tilde E_n$, depending on the divided differences of the  data,  might be considered a {\em realistic correction} for $Q_n$, in the sense  that $\tilde E_n$ should be close to the magnitude of the true error of $Q_n$, having also the correct sign. The analysis of the theoretical error of the rule $S_n$  as well as some classical properties for divided differences  suggest  the inclusion of one or two new points in the given panel.  When $n$ is even it is included one point and two points otherwise.  In both cases this approach enables the computation of a {\em realistic error} $\bar E_{S_n}$ for the  {\it extended or corrected} rule $S_n$. The respective output $(Q_n,\tilde E_n, S_n, \bar E_{S_n})$ contains reliable information on the quality of the approximations $Q_n$ and $S_n$, provided certain conditions involving ratios for the derivatives of the function $f$ are fulfilled.  These simple rules are easily converted into {\it composite} ones.  Numerical examples are presented showing that these quadrature rules are useful as a computational alternative to the classical Newton\--Cotes formulas.

   \end{abstract}



\section{Introduction}\label{introd}
 The first  two  quadrature  rules taught  in any numerical analysis course  belong to a group  known as closed Newton\--Cotes rules.  They are  used to approximate the integral   $I(f)=\int_a^b f(x) dx$  of a sufficiently smooth  function $f$ in the finite interval $[a,b]$.  The basic rules are known as   trapezoidal rule and  the Simpson's rule.  The trapezoidal rule is   $Q(f)=h/2\, (f(a)+f(b))$,   for which  $h=b-a$,  and has  the theoretical  error
 \begin{equation}\label{nova1}
\begin{array}{l}
I(f)-Q(f)= - \displaystyle{\frac{h^3}{12}}\, f^{(2)}(\xi), 
\end{array}
\end{equation}
while  the Simpson's rule is  $Q(f)=h/3\, (f(a)+4\, f((a+b)/2)+ f(b))$, with  $h=(b-a)/2$,  and  its error   is
\begin{equation}\label{nova1a}
\begin{array}{l}
I(f)-Q(f)=- \displaystyle{\frac{h^5}{90}} \, f^{(4)}(\xi). 
\end{array}
\end{equation}
The error formulas  \eqref{nova1} and \eqref{nova1a} are of existential type.  In fact, assuming that   $f^{(2)}$ and $f^{(4)}$  are (respectively)  continuous,  the expressions  \eqref{nova1} and \eqref{nova1a}  say that there exist a point $\xi$,  somewhere in the interval $(a,b)$, for which the respective  error has the displayed form.  From a computational point of view the utility of  these error expressions is rather limited since in general  is  quite difficult or even impossible to obtain expressions for the derivatives  $f^{(2)}$ or $f^{(4)}$, and consequently bounds for  $|I(f)-Q(f)|_\infty$. Even in the case one obtains such bounds  they  generally overestimate the true error of $Q(f)$.

\noindent
Under mild assumption on the smoothness of the integrand function $f$, our aim is to determine certain quadrature rules, say   $R(f)$, as well as approximations for its error  $\tilde E(f)$,  using only the information contained in the table of values or {\it panel} arising from the discretization of the problem. The algorithm to be constructed will produce the numerical value  $R(f)$, the correction or estimated error $\tilde E(f)$ as well as  the value of the {\it interpolatory} rule
\begin{equation}\label{nova2}
S(f)=R(f)+\tilde E(f).
\end{equation} 
The true error of $S(f)$ should be much less than the estimated error of $R(f)$, that is,
\begin{equation}\label{erromenor}
| I(f)-S(f)| << |\tilde E(f)|,
\end{equation}
for a sufficiently small step  $h$.
In such case we say that $\tilde E(f)$ is a {\em realistic correction} for $R(f)$. Unlike the usual approach where one builds  a quadrature formula $ Q (f) $ (like the trapezoidal or Simpson's rule) which is supposed to be a reasonable approximation to the exact value of the integral,  here we do not care wether the approximation $ R (f) $ is eventually  bad, provided that the  correction $ \tilde E (f) $ has been well modeled. In this case the value $ S (f) $ will be a good approximation to the exact value of the integral $ I (f) $. Besides the values $R(f)$, $\tilde E(f)$ and $S(f)$ we are also interested in computing a good estimation $\bar E_{S}(f) $ for the true error of $S(f)$, in the following sense. If the true error $E(f)=I(f)-S(f)$ is expressed in the standard decimal form as $E(f)=\pm 0.d_1d_2\cdots d_m\times 10 ^{-k},\,\, k\geq 0$, the approximation  $\bar E (f)$ is said to be {\em realistic} if its decimal form has the same sign as $E(f)$ and its first digit  in the mantissa differs at most one unit, that is, $\bar E(f)=\pm 0. (d_1\pm 1)\, \cdots\,  \times 10^{-k}$ (the dots represent any decimal digit). Finally, the algorithm to be used will produce the values $(R(f), \tilde E(f), S(f), \bar E(f))$.

\bigbreak
\noindent
In section \ref{subsec1} we present two models for building simple quadrature rules named  model $A$ and model $B$. Although both models are derived from the same method, in this work we focus our attention mainly on the model $A$. Definitions, notations and background material are presented in section \ref {main}. In Proposition \ref{pesosA} we obtain  the {\it weights} for the quadrature rule in model $A$ by  the undetermined coefficients method as well as  the  theoretical error expressions for the rules  are deduced  (see Proposition \ref {pCA}). The main results are discussed in Section \ref{secmodeloA}, namely  in Proposition \ref {prA} we show  that a reliable computation of realistic errors depends on the behavior of a certain  function involving ratios between high order  derivatives of the integrand function $ f $ and its first derivative.

\noindent
Composite rules for model A are presented in Section \ref {seccompositeA}  where some numerical examples  illustrate  how our approach  allows to obtain realistic error's estimates  for these rules.

\bigbreak

\subsection{ Two models}\label{subsec1}

\noindent
In this work we consider to be given a {\it panel} of  $n\,(n\geq 2)$  points $ \{(x_1, f_1), (x_2, f_2), $ $\ldots, (x_n, f_n) \} $, in the interval $ [a, b] $, having the nodes $ x_i $  equally spaced with step $ h> 0 $,  $f_i=f(x_i)$, where $f$  a sufficiently smooth function in the interval.  We consider the following two models:

\bigbreak
\underline{Model A}
\medskip

\noindent
Using only the first node of the panel we construct a quadrature rule  $Q_n(f)$ adding a correction  $\tilde E_n(f)$,  so that the   {\it corrected} or  {\it extended}  rule  $S_n(f)= Q_n(f)+\tilde E_n(f)$ 
is {\em interpolatory} for the whole panel,
\begin{align}\nonumber
S_n(f)&= Q_n(f)+\tilde E_n(f)\\ \label{modeloA}
&= a_1\, f(x_1)+ \left\{ a_2\,f[x_1,x_2]+\cdots+a_n\, f[x_1,x_2,\ldots,x_n]  \right\},
\end{align}
where $f[x_1,x_2,\ldots,x_n]$ denotes the $(n-1)$\--th divided difference  and $a_1$, $a_2$, $\ldots$, $a_n$ are   {\it weights} to be determined.

\noindent
Note that $Q_n(f)$ is simply the so\--called   {\it left rectangle} rule, thus $\tilde E_n(f)=\sum_{j=2}^{n}$ $ a_j\, f[x_1,\ldots,x_j]$  can  be seen  as a correction to such a rule.

\bigbreak
\underline{Model B}
\medskip

\noindent
The rule $Q_n(f)$ uses the first $n-1$ points of the panel (therefore is  not interpolatory in the whole panel),  and it is added  a correction term $\tilde E_n(f)$,  so that the  {\it corrected} or  {\it extended}  rule  $S_n(f)$ is interpolatory,
\begin{align}\nonumber
S_n(f)&= Q_n(f)+\tilde E_n(f)\\ \label{modeloB}
&= \left\{\,a_1\, f(x_1)+ a_2\,f(x_2)+\cdots+a_{n-1}\, f(x_{n-1})\,  \right\}+a_n\,f[x_1,x_2,\ldots,x_n]. 
\end{align}

\noindent
Since the interpolating polynomial of the panel is unique,  the value computed for $S_n(f)$ using either model is the same and equal to the value one finds if the simple  closed Newton\--Cotes rule for $n$ equally spaced points has been applied to the data. This means that the extended rules \eqref{modeloA} and \eqref{modeloB} are both  algebraically equivalent to the referred  simple Newton\--Cotes rules. However, the algorithms associated to each of the models  \eqref{modeloA} and \eqref{modeloB} are not the  classical ones for the referred rules.  In particular,  we can show that that for $n$ odd,  the rules $Q_n(f)$ in model B are open Newton\--Cotes formulas \cite{mg1}. Therefore, the extended  rule $S_n$ in model B can be seen as a bridge between open and closed Newton\--Cotes rules.

\noindent
The  method of undetermined coefficients applied to a Newton's basis of polynomials is used  in order to obtain $S_n(f)$.  The associated system of equations is diagonal,   The same method can also be applied  applied to get any   hybrid model obtained from the models A and B.  For instance,  an hybrid extended rule using  $n=3$ points could be written as
$$
S_3(f)=\left\{  a_1\, f(x_1)+ a_2\, f(x_2)+ a_3\, f(x_3)  \right\} + a_4\, f[x_1,x_2]+ a_5\, f[x_1,x_2,x_3].
$$
In this work our study is mainly  focused in model $A$.

 \section{Notation and background}\label{main}
 
 \begin{defn}\label{defstandard} (Canonical and Newton's basis)

\noindent
Let  ${\cal P}_k$  be the vector space of real polynomials   of degree less or equal to $k$ ($k$ a nonnegative integer).  The set  $\,<\phi_0(x),\phi_1(x),\-\ldots, \phi_{n-1}(x)>$, where
\begin{equation}\label{stand1}
\phi_i(x)= x^i,\quad i=0,1,\ldots, n-1,
\end{equation}
is the canonical basis for  ${\cal P}_{n-1}$.

\noindent
Given  $n\geq 1$ distinct points $x_1,x_2,\ldots x_n$,   the set  $<w_0(x),w_1(x),\ldots, w_{n-1}(x)>$, with
\begin{equation}\label{stand2}
\begin{array}{l}
w_0(x)=1\\
w_i(x)=w_{i-1} (x)\times (x-x_i), \quad i=1,\ldots,(n-1)
\end{array}
\end{equation}
 is known as the Newton's basis for ${\cal P}_{n-1}$.
\end{defn}

\noindent  
 A polynomial interpolatory quadrature rule $R_n(f)$ obtained from a given panel $\{ (x_1,f_1),\- (x_2, f_2),\ldots, (x_n,f_n)\}$, where $x_i\neq x_j$ for $i\neq j$,  has the form 
\begin{equation}\label{R1}
   R_n(f)=c_1\, f(x_1)+c_2\, f(x_2)+\cdots+c_n f(x_n), 
\end{equation}
     where the coefficients  (or weights) $c_j\,\,(1\leq j\leq n)$ can be computed assuming  the quadrature rule is exact for any polynomial $q$ of degree less or equal to $n-1$, that is, $deg(R_n(f))=n-1$, according to the following definition
     
      \begin{defn}\label{defdegree}  (Degree of exactness) (\cite{gautschi}, p. 157)
 
 \noindent
 A quadrature rule $R_n(f)=\sum_{i=1}^{n} c_i\, f(x_i)$ has (polynomial) {\it degree of exactness} $d$ if the rule is exact whenever $f$ is a polynomial of degree $\leq d$, that is
 $$
 E_n(f)=I(f)-R_n(f)=0\quad \mbox{for all}\quad f\,\in {\cal P}_d.
 $$
 The degree of the quadrature rule is denoted by $deg(R)$. When $deg(R)=n-1$ the rule is called interpolatory.
 \end{defn}

     \noindent
      In  particular for a $n$\--point panel  the interpolating polynomial $p$  satisfies
\begin{equation}\label{eq1}
   f(x)=p(x)+r(x),\qquad x\,\in [a,b],
\end{equation}
   where  $p$ can be written in Newton's form   (see for instance \cite{steffensen}, p. 23, or any standard text in numerical analysis)
\begin{equation}\label{eq1nova}
   p(x)=f_1+f[x_1,x_2]\,(x-x_1)+\cdots+f[x_1,x_2,\ldots,x_n]\,(x-x_1)\,(x-x_2)\cdots (x-x_{n-1})\\
\end{equation}
   and the remainder\--term is
 \begin{equation}\label{eq1novaA}
   r(x)=f[x_1,x_2,\ldots,x_n,x]\,  (x-x_1)\,(x-x_2)\cdots (x-x_{n}),
\end{equation}
where  $f_i=f(x_i)$ and  $f[x_1,x_{2},\ldots, x_{1+j}]$, for $j\geq 0$, denotes the $j$\--th  divided difference of the data $(x_i,f_i)$, with  $i=1,\ldots, n$. 
Therefore  from \eqref{eq1} we obtain
\begin{equation}\label{eq1novaC}
I(f)=\int_a^b p(x) dx+ \int_a^b r(x) dx= R_n(f)+ E_{R_n}(f),
\end{equation} 
\noindent
 where $ E_{R_n}(f)$ denotes the true error of the rule $R_n(f)$. Thus,  
\begin{equation}\label{ruleX}
\begin{array}{ll}
R_n(f)&=\int_{x_1}^{x_n} w_0(x) dx\times f_1+\int_{x_1}^{x_n} w_1(x) dx\times f[x_1,x_2]+\cdots+\\
&+\int_{x_1}^{x_n} w_{n-1}(x) dx\times f[x_1,x_2,\ldots,x_n]
\end{array}
\end{equation}
The expression  \eqref{ruleX} suggests that the application of the undetermined coefficient method using the Newton's basis for polynomials should be rewarding since the successive divided differences are trivial for such a basis. In particular,  the weights for  the extended rule in model A are trivially computed.

\begin{prop}\label{pesosA}
The weights for the rule $S_n(f)$ in model A are  
\begin{equation}\label{pesosparaA}
a_i=I(w_{i-1})=\int_0^{(n-1)\,h} \, w_{i-1}(t) dt\qquad  i=1,2,\ldots,n
\end{equation}

\end{prop}
\begin{proof}

\noindent
The divided differences do not depend on a particular node but on the distance between nodes. Thus  for any given $n$\--point panel of constant step $h$, we can assume without loss of generality that $x_1=0,x_2=h,\ldots, x_n=(n-1)\, h$.  Considering  the Newton's basis for polynomials $w_0 (t)=1$, $w_1(t)=t$, $\ldots$, $w_{n-1} (t)= t\,(t-h)\cdots (t-(n-2)\,h)$, where $0\leq t\leq (n-1)\,h$,  from \eqref{modeloA} we have
$$
S_n (w_0)=a_1, \quad S_n(w_1)=a_2,\ldots, S_n(w_i)= a_{i+1}\quad \mbox{for} \quad i=0,\ldots,(n-1).
$$
The undetermined coefficients method  applied to the Newton's basis $<w_0(t)$,  $w_1(t) $, $\ldots,w_{n-1}(t)>$ leads to the $n$ conditions $a_i=I(w_{i-1})$ or, equivalently, to a diagonal system of linear equations whose matrix is the identity. The equalities in \eqref{pesosparaA} can also be obtained directly from \eqref{ruleX}.
\end{proof}
  \noindent
Theoretical expressions for the error $E_{R_n}(f)$ in \eqref{eq1novaC} can be obtained  either  via  the mean value theorem for integrals or by  considering the so\--called  Peano kernel (\cite {davis}, p. 285, \cite{ gautschi}, p. 176).  However, we will use the method of undetermined coefficients  whenever theoretical expressions for the errors $E_{Q_n}$ and $E_{S_n}$ are needed.

\bigbreak
\noindent
For sufficiently smooth functions $f$, the fundamental relationship between divided differences on a given panel and the derivatives of $f$ is given by the following well known result,
  
 \begin{prop}\label{difderivA} (\cite{steffensen}, p. 24),(\cite{gautschi} p. 101]), (\cite{krylov}, p. 41)
 
 \noindent
Given $n\,\,(n\geq 2)$ distinct nodes $\{x_1,\ldots,x_n\}$ in $ J=[a,b] $, and $f\in C^{n-1}(J)$, there exists $\xi \in\,(x_1,x_n)$ such that 
\begin{equation}
 \begin{array}{l}\label{difderiv}
  f[x_1,x_2,\ldots,x_{n}]=\displaystyle{\frac{f^{(n-1)}(\xi)}{(n-1)!}}.
\end{array}
\end{equation}
\end{prop}

\noindent
Applying \eqref{difderiv} to the canonical or Newton's basis, we get
\begin{equation}\label{krylovB}
\begin{array}{l}
\phi_j[x_1,x_2,\ldots,x_n]=w_j[x_1,x_2,\ldots,x_n]=0,\qquad \mbox{for}\quad j=0,1,\ldots,(n-2)\\
\phi_{n-1}[x_1,x_2,\ldots,x_n]=w_{n-1}[x_1,x_2,\ldots,x_n]=1.
\end{array}
\end{equation}

\noindent
By construction, the rules $S_n$ in models $A$ and $B$ are at least of degree $n-1$ of precision according to Definition \ref{defdegree}.
 
 \bigbreak
 \noindent
   In this work the undetermined coefficients method enables us to obtain both the  weights and theoretical error formulas. This apparently contradicts the following assertion due to Walter Gautschi (\cite{gautschi}, p. 176):  \lq\lq{The method of undetermined coefficients, in contrast, generates only the coefficients in the approximation and gives no clue as to the approximation error}\rq\rq. 

\bigbreak
\noindent
Note that by Definition \ref{defdegree} the theoretical error  \eqref{nova1} says  that $deg(Q)=1$ for the trapezoidal rule,  and  from  \eqref{nova1a} one  concludes  that $deg(Q)=3$ for the Simpson's rule. This suggests the following assumption.
 
 \begin{axiom}\label{gautshinova}
Let be given a $n$\--point ($n\geq 1$) panel with constant step $h>0$,  a sufficiently smooth function $f$ defined on the interval $[a,b]$, and a quadrature rule  $R(f)$  (interpolatory or not)  of degree $m$,  there exists a constant $K_h\neq 0$ (depending on a certain power of $h$) and a point $\xi$, such that

 \begin{equation}\label{lab1}
E(f)=I(f)-R(f)= K_h\, f^{(m+1)}(\xi), \qquad \xi\in (a,b)
\end{equation}
where de derivative $f^{(m+1)}$ is not identically null in $[a,b]$, and $m$ is the \it {least} integer for which \eqref{lab1} holds.
 \end{axiom}
 
 \noindent
 The expression \eqref{lab1} is crucial in order to deduce formulas for the theoretical error of  the rules in model $A$ or $B$.
 
 \begin{prop}\label{corgaut} Under  Assumption \ref{gautshinova}, the constant $K_h$ in \eqref{lab1} is
 \begin{equation}\label{lab2}
K_h=\displaystyle{\frac{I(\omega_{m+1})-R(\omega_{m+1})}{(m+1)!}},
\end{equation}
where $\omega_{m+1}(x)$ is either the element  $\phi_{m+1}(x)$  of the canonical basis, or the element  $w_{m+1}(x)$ of the Newton's basis, or any polynomial of degree $m+1$ taken from any basis for polynomials used to apply the undetermined coefficients method.
 \end{prop}
\begin{proof} 

\noindent
For any nonnegative integer $i\leq m$, from  \eqref{lab1} we get $E(\omega_i)=I(\omega_i)-R(\omega_i)= K_h\times 0$. As $m$ is the least integer for which the righthand side of \eqref{lab1} is non zero, and $deg(R)=m$, one has
$$
I(\omega_{m+1})-R(\omega_{m+1})=K_h\times \omega_{m+1}^{(m+1)} (x) =K_h\times (m+1)! ,
$$
from which it follows \eqref{lab2}. 
\noindent
As for a fixed  basis the interpolating polynomial is unique,  it follows that the error for the corresponding  interpolatory rule is unique as well.
\end{proof}

\noindent
In Proposition \ref{pCA} we  show that  $deg(S_n(f))$ depends on the parity of $n$. Thus, we recover a well known result about the precision of the Newton\--Cotes rules, since $S_n(f)$ is algebraically  equivalent to a closed $Newton\--Cotes$ formula with $n$ nodes.  Let us first prove the following lemma.
\begin{lemma}\label{lema1}
Consider the Newton's polynomials
$$
\begin{array}{l}\label{w}
w_0(t)=1\\
w_1(t)=t\\
w_j(t)=\prod_{i=1}^{j-1} t\, (t-i\,h),\quad j\geq 2.
\end{array}
$$
Let $n\geq 1$ be an integer and  $I^{[-1]},I^{[0]}$ and $I^{[-2]}$ the following integrals:
$$
\begin{array}{l}
I^{[-1]}(w_n)=\int_{0}^{(n-1)\,h} w_n(t) dt\\
I^{[0]}(w_n)=\int_{0}^{n\,h} w_n(t) dt\\
I^{[-2]}(w_n)=\int_{0}^{(n-2)\,h} w_n(t) dt.\\
\end{array}
$$
Then, 

\noindent
(a) $I^{[-1]}(w_n)\neq 0$ for $n$ even and  $I^{[-1]}(w_n)= 0$  for $n$ odd;

\noindent
(b) $I^{[0]}(w_n)\neq 0$  for $n\geq 2$ and  $I^{[-2]}(w_n)\neq 0$  for $n\neq 3$.

\end{lemma}
\begin{proof}

\noindent
(a) For $n=1$ it is obvious that $I^{[-1]}(1)=0$. For  any integer $n\geq 2$,  let us change  the integration interval $[0,(n-1)\, h]$ into the interval ${\cal I}=[-(n-1)/2,(n-1)/2]$ and consider the bijection $x=\omega(t)=h\,(x+(n-1)/2)$.
For $n$ odd, we obtain
\begin{align*}
I^{[-1]}(w_n)&= \int_{0}^{(n-1)\,h} w_n(t) dt\\
&= h^{n+1}\int_{-(n-1)/2}^{(n-1)/2} \left[\left(t^2-\left( \frac{n-1}{2}\right)^2\right)\,\left(t^2- \left(\frac{n-1}{2}-1\right)^2\right)\cdot\right.\\
&\left.\qquad\cdot \left(t^2- \left(\frac{n-1}{2}-2\right)^2\right)\cdots  \left(t^2- 1\right)\, t\,\right]\,dt .
\end{align*}

\noindent
As the integrand  is an odd function in  ${\cal I}$, we have  $I^{[-1]}(w_n)=0$. 

\noindent
For $n$  even, we get
\begin{align*}
I^{[-1]}(w_n)&= \int_{0}^{(n-1)\,h} w_n(t) dt\\
&= h^{n+1}\int_{-(n-1)/2}^{(n-1)/2}\left[\left(t^2- \frac{(n-1)^2}{4}\right)\,\left(t^2- \frac{(n-3)^2}{4}\right)\cdots\right.\\
&\qquad\left. \cdots\left(t^2- \frac{(n-5)^2}{4}\right)\cdots  (t^2- 1/4)\right] \,dt,
\end{align*}
where the integrand is an even function, thus $I^{[-1]}(w_n)\neq 0$. 

\noindent
(b) The proof is analogous so it is ommited.
\end{proof}

\noindent
The degree of precision for the rules in models A and B, and  the respective true errors can be easily obtained using the undetermined coefficients method, the Lemma \ref{lema1} and Proposition \ref{corgaut}. The next propositions ( \ref{pCA}, \ref{pcAA} and  \ref{propA})  establish the theoretical errors and degree of precision of these rules. In particular, in  Proposition \ref{pCA} we recover a classical result on the theoretical error for the rule $S_n(f)$ \--- see for instance  \cite{isaacson}, p. 313.
\begin{prop}\label{pCA}
Consider the rule $S_n(f)$  for the models $A$ or $B$ defined in the panel $\{ (x_1,f_1),\ldots,(x_n,f_n)  \}$, and assume that $J=[a,b]$ is a finite interval containing the nodes $x_1,\ldots,x_n$, for $n\geq 2$. Let $w_n(x)$ denote the Newton's polynomial of degree $n$.  The respective degree of precision and true error $E_{S_n}=I(f)-S_n(f)$ are the following:

\noindent
(i) If $n$ is odd and $f\in C^{n+1}(J)$, then
\begin{equation}\label{ca1}
\begin{array}{l}
deg(S_n)=n\\
E_{S_n}=\displaystyle{ \frac{ I(w_{n+1})}{ (n+1)!}   } \, f^{(n+1)} (\xi), \qquad \xi\in J.
\end{array}
\end{equation}

\noindent
(ii) If $n$ is even and $f\in C^{n}(J)$, then
\begin{equation}\label{ca2}
\begin{array}{l}
deg(S_n)=n-1\\
E_{S_n}=\displaystyle{ \frac{ I(w_{n})}{ n!}   } \, f^{(n)} (\xi), \qquad \xi\in J.
\end{array}
\end{equation}
\end{prop}
\begin{proof}
 We can assume without loss of generality  that the panel is $\{ (0,f_1)$, $(h,f_2),\ldots$, $((n-1)h, f_n)\}$ (just translate the point $x_1$) . For $n$ even or odd, by construction of the interpolatory rule $S_n(f)$, we have $deg(S_n)\geq n-1$ in model $A$ or $B$. Taking the Newton's polynomial $w_n(t)=t (t-h)\cdots (t-(n-1)h))$, whose zeros  are $0,h,\ldots, (n-1)\,h$, we get for the divided differences in \eqref{modeloA},  
$$
w_n(0)=w_n[0,h]=\ldots=w_n[0,h,\ldots,(n-1)h]=0,
$$
where $f$ has been substituted by $w_n$ in \eqref{modeloA}. Thus, $S_n(w_n)=0$ and $S_j(w_n)=0$ for any $j\geq n+1$.

\noindent
(i) For $n$ odd, by Lemma \ref{lema1} (a) we get $I(w_n)=0$. Thus $S_n(w_n)=I(w_n)$, and so  $deg(S_n)\geq n+1$. However by Lemma \ref{lema1} (b) we have $I(w_{n+1})=\int_{0}^{(n-1)\,h} w_{n+1} (t) dt\neq 0$.  Therefore, $deg(S_n)=n$ and so \eqref{ca1} follows by  Proposition \ref{corgaut}.
\end{proof}

\begin{prop}\label{pcAA}
 Consider the rule $Q_n(f)=I(w_0)\, f(x_1)$ given by model A, and assume that $f\in C(J)$, where $J=[a,b]$ is a finite interval containing the nodes $x_1,x_2,\ldots,  x_n$, for  $n\geq 2$. Then, there exists a point $\xi\in J$ such that
\begin{equation}\label{eqpcAA1}
E_{Q_n} (f)=I(f)-Q_n(f)= I(w_1)\, f'(\xi)=\displaystyle{ \frac{(n-1)^2\, h^2}{2}   }\, f' (\xi).
\end{equation}
\end{prop}
\begin{proof} Taking  $x_1=0$ and $w_1(t)=t$, we have $Q_n(w_1)=0$ and $I(w_1)\neq 0$. Thus, by Proposition \ref{corgaut},  we obtain the equalities in \eqref{eqpcAA1}.
\end{proof}
\begin{prop}\label{propA}
Consider the rule $Q_n(f)$ given in model B and assume that $f\in C^{n}(J)$, where $J=[a,b]$ is a finite interval containing the nodes $x_1,\ldots,x_n$, for  $n\geq 2$. Let $w_{n-1}(x)$ be the Newton's polynomial of degree $n-1$. The degree of precision for $Q_n$ is $n-2$ and there exists a point $\xi\in J$  such that
\begin{equation}\label{eqA1}
E_{Q_n}(f)=\displaystyle{  \frac{I(w_{n-1})}{(n-1)!}   }\, f^{(n-1)}(\xi).
\end{equation}
\end{prop}
\begin{proof} 

 Without loss of generality consider the panel  $\{ (0,f_1)$, $(h,f_2)$, $\ldots$, $((n-1)h, f_n)\}$. By construction, via the undetermined coefficients,  we have  $deg(Q_n)\geq (n-2)$. Taking $w_{n-1}(t) = t\, (t-h)\ldots (t-(n-2)h)$, we have $w_{n-1}(t_i)=0$, for $i=0,\ldots,(n-2)$, so $Q(w_{(n-1)})=0$. By Lemma \ref{lema1} (b) $I(w_{n-1})=\int_0^{(n-1)h} w_{n-1}(t) dt\neq 0$ and therefore $m=deg(Q_2)=n-2$. Thus, by Proposition \ref{difderivA}, there exist $\theta\in (0,(n-1)h)$ such that,
$$
E_{Q_n}(f)=I(f)-Q_n(f)= \displaystyle{  \frac{I(w_{n-1})-Q_{n-1}}{(m+1)!}   }\, f^{(m+1)}(\theta)=\displaystyle{  \frac{I(w_{n-1})}{(n-1)!}   }\, f^{(n-1)}(\theta).
$$
To the point $\theta$ it corresponds a point $\xi$ in the interval $J$, and so \eqref{eqA1} holds.
 \end{proof}

\section{Realistic errors for model A}\label{secmodeloA}
 The properties discussed in the previous Section are valid  for both models $A$ and $B$. However here we will only present some numerical examples for the rules defined by model $A$. A detailed discussion and examples for  model $B$ will be presented elsewhere.
\bigbreak
\noindent
 From \eqref{pesosparaA} we obtain immediately the weights for any rule of $n$ points  defined by model A. Such weights are displayed in Table \ref{tabelamodeloA}, for $2\leq n\leq 9$. The values displayed should be multiplied by an appropriate power of $h$ as indicated in the table's  label. According to Proposition \ref{pCA}, the last column in this table contains the value $d=deg(S_n)$ for the degree of precision of the rule $S_n(f)$.

  \begin{table}[h!]
   $$
   \hspace{-0.cm}
   \begin{array}{| c |  c  |  c   |  c   |   c  |   c | c   |   c  |   c  |   c  |   l  |   }
   \hline
   n&\begin{array}{l}  a_1 \end{array} &  \begin{array}{l}  a_2 \end{array}  & \begin{array}{l}  a_3 \end{array}  & \begin{array}{l}  a_4 \end{array}  & \begin{array}{l}  a_5 \end{array} &\begin{array}{l}  a_6\end{array}  &  \begin{array}{l}  a_7\end{array}  & \begin{array}{l}  a_8\end{array} &   \begin{array}{l}  a_9 \end{array}  & d\\
   \hline
   2&1&\frac{1}{2} &     &     &     &    &      &   &   &  1\\
   \hline
   3&2&2& \frac{2}{3}&    &    &     &    &   &  &  3\\
   \hline
   4&3&\frac{9}{2}&  \frac{9}{2} & \frac{9}{4} &   &  &  &    &  & 3\\
   \hline
   5&4&8&\frac{40}{3}& 16 & \frac{112}{15} &   &   &  &  & 5 \\
   \hline
   6&5&\frac{25}{2}&\frac{175}{6}&\frac{225}{4}&\frac{425}{6}&\frac{475}{12}&   &  &  & 5 \\
   \hline
   7&6&18&54&144&\frac{1476}{5}& 396& \frac{1476}{7}    &    &     & 7\\
   \hline
   8&7&\frac{49}{2}&\frac{539}{6}&\frac{1\,225}{4}&\frac{26\,117}{30}&\frac{7\,497}{4}&\frac{30\, 919}{12}&\frac{36\,799}{24}  &  & 7\\
   \hline
   9&8&32&\frac{416}{3}&576&\frac{31\,424}{15}&\frac{18\,688}{3}&\frac{290\,048}{21}&\frac{58\,880}{3}& \frac{506\,368}{45} & 9 \\
   \hline
     \end{array}
   $$
   \caption{\label{tabelamodeloA}  Weights for model A,  $ a_j=\int_{0}^{(n-1) h} w_{j-1}(t) dt, \quad j=1,2,\ldots, n $.  The entries in column $2$ (heading $a_1$) should be multiplied by $h$; the entries in column $3$ (heading $a_2$) multiplied  by $h^2$, and so on.}
\end{table}

\noindent
Note that, by construction, the weights in model A are positive for any $n\geq 2$. Therefore, the respective extended rule $S_n(f)$ does not suffers from the inconvenient observed in the traditional form for Newton\--Cotes rules where, for $n$ large  $(n\geq 9)$  the weights are of mixed sign  leading eventually  to losses of significance by cancellation.

\bigbreak
\noindent
The next Proposition \ref{prA} shows that a reliable computation of realistic errors for the rule $S_n(f)$, for $n\geq 3$, depends on  the behavior of  a certain function $g(x,h)$ involving certain quotients  between derivatives  of higher order of $f$ and its  first derivative.  Fortunately, in the applications, only a crude information on the function $g(x,h)$ is needed, and in practice  it will be sufficient to plot $g(x,h)$ for some different values of the step $h$, as it is illustrated in the numerical examples given in this Section (for some simple rules) and in Section \ref{seccompositeA} (for some composite rules) .

\begin{prop}\label{prA}
Consider a panel of $n\geq 2$ points and the model A for approximating $I(f)=\int_a^b f(x) dx$,  where $f$ is a sufficiently smooth function defined in the interval $J=[a,b]$ containing the panel nodes. Let
$$
\begin{array}{ll}
S_n(f)&= Q_n(f)+\tilde E_n(f)\\
&= a_1\, f(x_1)+\left\{   \sum_{k=2}^{n} a_k\, f[x_1,x_2,\ldots. x_k] \right\},
\end{array}
$$
where $a_i= I(w_{i-1}),\,\, i=1,2,\ldots,n$. Denote by $g(x,h)$ (or $g(x)$ when $h$ is fixed) the function
\begin{equation}\label{cond1A}
g(x,h)= \left|1+    \sum_{j=2}^{n-1} \displaystyle{ \frac{a_{j+1}}{a_2}  }    \,   \displaystyle{  \frac{f^{(j)}(x)}{j!\, f'(x)} }\right| .
\end{equation}

\noindent 
Assuming that
\begin{equation}\label{cond1}
 f'(x)\neq 0\quad \forall x\in \, (x_1,x_n)
 \end{equation}
 and
 \begin{equation}\label{cond2}
g(x,h)\geq h \qquad \forall x\in\, (x_1, x_n)  
\end{equation}
then, for a sufficiently small step $h>0$, the correction $\tilde E_n(f)$ is  realistic  for $Q_n(f)$. Furthermore, the true error of $S_n(f)$ can be estimated by the following realistic errors:

\noindent
(a) For $n$ \underline{odd}:
\begin{equation}\label{pnA}
\bar E_{S_n}=\displaystyle{  \frac{I(w_{n+1})}{I(w_1)} }\, \displaystyle{\frac{f[x_1,x_2,\ldots,x_n,\bar x_1,\bar x_2]}{f[x_1,x_2]}}\times \tilde E_n(f),
\end{equation}
where $\bar x_1=(x_1+x_2)/2$ and $\bar x_2=(x_{n-1}+x_n)/2$.

\noindent
(b) For $n$ \underline{even}:
\begin{equation}\label{pnB}
\bar E_{S_n}=\displaystyle{  \frac{I(w_{n})}{I(w_1)} }\, \displaystyle{\frac{f[x_1,x_2,\ldots,x_n,\bar x]}{f[x_1,x_2]}}\times \tilde E_n(f),
\end{equation}
where $\bar x=(x_1+x_2)/2$.
\end{prop}
\begin{proof}

\noindent
(a)  By Proposition \ref{pCA} (i), we have
$$
I(f)-S_n(f)=\displaystyle{  \frac{I(w_{n+1})}{(n+1)!} }\,  f^{(n+1)} (\xi),\qquad \xi\in (x_1,x_n),
$$
and from Proposition \ref{difderivA} the correction $\tilde E_n(f)$ can be written as
$$
\tilde E_n(f)= a_2\, f'(\theta)+  a_3\, \displaystyle{ \frac{f^{(2)} (\theta_1)}{2!}  }  + \ldots +  a_n\, \displaystyle{ \frac{f^{(n-1)} (\theta_{n-2})}{(n-1)!}  },
$$
where $\theta\in (x_1,x_2)$, $\theta_1\in (x_1,x_3)$, $\ldots$, $\theta_{n-2}\in (x_1,x_2,\ldots,x_n)$. Therefore,
$$
\begin{array}{ll}
\tilde E_n(f) & =a_2 \,f'(\theta) \left( 1+\sum_{j=2}^{n-1}   \displaystyle{ \frac{a_{j+1}}{a_2}  }  \displaystyle{  \frac{f^{(j)} (  \theta_{j-1} )}{j!\, f'(\theta )}    }     \right)\\
&=    I(w_1) \,f'(\theta) \left( 1+\sum_{j=2}^{n-1}  \displaystyle{ \frac{a_{j+1}}{a_2}  }    \displaystyle{  \frac{f^{(j)} (  \theta_{j-1} )}{j!\, f'(  \theta)}    }     \right).
\end{array}
$$
 Thus, using the hypothesis in \eqref{cond1} we obtain 
$$
\begin{array}{ll}
\displaystyle{  \frac{I(f) -S_n(f)}{\tilde E_n(f)} }=\displaystyle{  \frac{I(w_{n+1}) }{(n+1)!\, I(w_1)}   }    \, \displaystyle{ \frac{f^{(n+1)} (\xi)}{ f'(\theta)}    }\, \displaystyle{   \frac{1}{ \left( 1+\sum_{j=2}^{n-1}  \displaystyle{ \frac{a_{j+1}}{a_2}  }   \displaystyle{  \frac{f^{(j)} (  \theta_{j-1} )}{j! \, f'(  \theta)}    }     \right)}      },
\end{array}
$$
 that is,
\begin{equation}\label{pn1}
\displaystyle{  \frac{I(f) -S_n(f)}{\tilde E_n(f)} }=\displaystyle{ \frac{c\, h^n}{(n+1)!} }    \, \displaystyle{ \frac{f^{(n+1)} (\xi)}{ f'(\theta)}    }\, \displaystyle{   \frac{1}{ \left( 1+\sum_{j=2}^{n-1}  \displaystyle{ \frac{a_{j+1}}{a_2}  }  \displaystyle{  \frac{f^{(j)} (  \theta_{j-1} )}{j! \, f'(  \theta)}    }     \right)}      },
\end{equation}
where $c$ is a constant not depending on $h$.
Thus, by  \eqref{cond2}  we get
$$
| I(f)-S_n(f)| \leq  \displaystyle{\frac{c\, h^{n-1}}{(n+1)!}    }  \left|    \frac{f^{(n+1)} (\xi)}{f'(\theta)} \right| \, | \tilde E_n(f) |.
$$
Therefore, for $h$ sufficiently small, $| I(f)-S_n(f)|<< |\tilde E_n(f)|$, that is, $\tilde E_n(f)$ is a realistic correction for $Q_n(f)$. Furthermore,
\begin{equation}\label{pn2}
| I(f)-S_n(f)| \leq   \displaystyle{\frac{c\, h^{n-1}}{(n+1)!}    }   M \, | \tilde E_n(f) |, \quad \mbox{where}\quad M=max_{x\in J}  \displaystyle{  \frac{| f^{(n+1)} (x)  |   }{|  f'(x)   |}  }.
\end{equation}
Finally, by Proposition \ref{difderivA} and the continuity of the function $f^{(n+1)}(x)$, we know that
$$
\displaystyle{   \frac{ f^{(n+1)} (\xi)    }{ (n+1)! }   } \simeq f[x_1,x_2,\ldots,x_n,\bar x_1,\bar x_2].
$$
So, from \eqref{pn1} we obtain \eqref{pnA}.

\noindent
(b) The proof is analogous so it is omitted.
\end{proof}

\noindent
The next  proposition shows that Proposition \ref{prA} for the case $n=2$ leads to the rule $S_2(f)$ which is algebraically equivalent to  the trapezoidal rule, and when $n=3$  the rule $S_3(f)$  is algebraically equivalent to  the  Simpson's rule.

\begin{prop}\label{propA3}
Let $I(f)=\int_a^b f(x) dx$, $J=[a,b]$, and  $f\in C^2(J)$. Consider the simple extended left rectangle rule
\begin{equation}\label{novaprop1}
S_2(f)=Q_2(f)+\tilde E_2(f)=h\, f(x_1)+\displaystyle{\frac{h^2}{2}}  \, f[x_1,x_2],
\end{equation}
and  $\bar x=(x_1+x_2)/2$, where $x_1=a$ and $x_2=b$.
Assuming that

$$
\begin{array}{l}
  f'(x)\neq 0,\qquad \forall x\in \, (x_1,x_2),
\end{array}
$$
then $\tilde E_2(f)$ is a realistic error for $Q_2(f)$,  for $h=(b-a)$ sufficiently small.  A realistic approximation for the true error of $S_2(f)$ is
\begin{equation}\label{k0}
\begin{array}{ll}
\bar E_{S_2}(f) =  -\displaystyle{\frac{h}{3}}\, \displaystyle{   \frac{ f[x_1,x_2,\bar x]}{f[x_1,x_2]} }\times \tilde E_2(f) .
\end{array}
\end{equation}
\end{prop}
\begin{proof}  By  Proposition \ref{difderivA} there exists a point $\theta\,\in (x_1,x_2)$ such that $f[x_1,x_2]=f'(\theta)$, so $\tilde E(f)=h^2/2\, f'(\theta)$. Using Proposition \ref{pCA}, we know that
\begin{equation}\label{k01}
\begin{array}{ll}
I(f)-S_2(f)&= \displaystyle{\frac{I(w_2)}{2}} \, f^{(2)}(\xi),\qquad \xi\in (x_1,x_2)\\
&= - \displaystyle{\frac{h^3}{12}}  \, f^{(2)} (\xi).
\end{array}
\end{equation}
As by hypothesis $f'(\theta)\neq 0$, we get
\begin{equation}\label{k02}
\displaystyle{\frac{I(f)-S_2(f)}{\tilde E_2(f)}} = - \displaystyle{ \frac{h}{6}} \displaystyle{  \frac{f^{(2)} (\xi)}{f'(\theta)}    },
\end{equation}
and so
\begin{equation}\label{eq1A3}
|I(f)-S_2(f)|\leq \displaystyle{\frac{M\,\, h}{6}      }  \, | \tilde E_2(f)|,\quad \mbox{where} \quad M=max_{x\in [x_1,x_2]} \displaystyle{ \frac{|f^{(2)}(x)|}{|f'(x)|}     }.
\end{equation}
Therefore,  for a sufficiently small $h$ the correction $\tilde E_2(f)$ is  realistic for $Q_2(f)$.
Since $f'(\theta)\simeq f[x_1,x2]$ and   $\displaystyle{ \frac{f^{2}(\xi)}{2}\simeq f[x_1,x_2,\bar x]  }  $ we obtain \eqref{k0} from \eqref{k01}.
\end{proof}

\begin{exam}\label{exemplo1modeloA}  (A realistic  error  for  $S_2(f)$)
\bigbreak
\noindent

\noindent
Let $I(f)=\int_{0}^{0.1} \sqrt{x} \,dx$. The function $f(x)=\sqrt{x}$ is not differentiable at $x=0$. However $f'(x)\neq 0$ for $x>0$, and the result \eqref{k0} still holds.   The numerical results (for $6$ digits of precision) are:

$$
\begin{array}{l}
I(f)=0.0210819\\
Q_2(f)= h\, f(0)=0\\
\tilde E_2(f)=\displaystyle{ \frac{h^2}{2}} \, f[0,0.1]=0.0158114\\
S_2 (f)=Q_2(f)+\tilde E_2 (f)=0.0158114 .
\end{array}
$$
The true error for $S_2(f)$ is

$$
E_{S_2}(f)=I(f)-S_2(f)=0.00527046\, .
$$
By \eqref{k0} the realistic error is
$$
\bar E_{S_2}(f) = -\displaystyle { \frac{h}{3}  \, \frac{f[0,0.1,0.05]}{f[ 0,0.1  ]}         } \times \tilde E_2(f)=0.00436619\, .
$$
 Table \ref{tabtrap1} shows that the realistic error   $\bar E_{S_2}(f)$ becomes closer to the true error when one goes from the step $h$ to the step $h/2$.

  \begin{table}[h!]
   $$
   \begin{array}{ | c |  c |  c|  }
   \hline
   h& \bar E_{S_2}(f) & E_{S_2}(f)= I(f)-S_2(f)\\
   \hline
   0.1&0.00436619 &0.00527046\\
   \hline
   0.05&0.00154368&0.00186339\\
   \hline
   0.025&0.00054577&0.000658808\\
   \hline
      \end{array}
   $$
   \caption{\label{tabtrap1}  Realistic and true error for $S_2 (f)$. }
\end{table}

\end{exam}

\begin{prop}\label{propA4}
 Consider the model A for $n=3$,
\begin{equation}\label{p3Anova}
\begin{array}{ll}
S_3(f)&=Q_3(f)+\tilde E_3(f)\\
&=2\,h\, f(x_1)+\left\{ 2h^2   \, f[x_1,x_2]   + \displaystyle{\frac{2\, h^3}{3}} \, f[x_1,x_2,x_3]\right\},
\end{array}
\end{equation}
where $x_1=a$, $x_2=(a+b)/2$ and $x_3=b$. 
Let
\begin{equation}\label{p3AA}
g(x,h)=  \left|1+\displaystyle{  \frac{h}{6} \, \frac{f^{(2)}(x)}{f'(x)}  }\right|.
\end{equation}
Assuming that $f\in C^4[a,b]$, if 
\begin{equation}\label{p3A}
(i) \qquad  f'(x)\neq 0\qquad \forall x\in \,(x_1,x_3)\\
\end{equation}
and
\begin{equation}\label{p3B}
(ii) \qquad g(x,h)\geq h \qquad \forall x\in \, (x_1,x_3),
\end{equation}
then, for a sufficiently small  $h=(b-a)/2$,  $\tilde E_3(f)$ is a realistic correction for $Q_3(f)$. A realistic approximation to the true error of $S_3(f)$ is
\begin{equation}\label{eq2A4}
\bar E_{S_3} (f) =  -\displaystyle{ \frac{2\, h^3}{15}\, \frac{f[x_1,x_2,x_3,\bar x_1, \bar x_2]}{f[x_1,x_2]}} \times \tilde E_3 (f),
\end{equation}
where $\bar x_1=(x_1+x_2)/2$ and $\bar x_2=(x_2+x_3)/2$.
\end{prop}
 
\begin{proof}  As $I(w_4)=\int_0^{2h} w_4(t) dt=-4 h^5/15$ and $I(w_1)=\int_0^{2h} w_1(t) dt=2\, h^2$,  by  Proposition \ref{prA} (a) we obtain  \eqref{eq2A4}.
\end{proof}

\begin{exam}\label{exemplo2modeloA}  (A realistic  error  for  the  $S_3(f) $ rule)

\bigbreak
\noindent
Let $I(f)=\int_{0}^{2\, h} e^{-x^2}dx=-\displaystyle{  \frac{\sqrt{\pi}}{2}  } \, \mbox{erf}(2\,h)$,  with $h>0$. Since $ f'(x)=-2\,x\, e^{-x^2}\neq 0$,  the condition \eqref{p3A} holds with $x_1=0$ and $x_3=2 h$.
Consider
$$
g(x,h)=\left|  1+ \frac{h}{6}   \frac{f^{(2)}(x)}{f'(x)} \right| = \frac{1}{6}\,  \left| 6+ h \left( \frac{1}{x}-2 x  \right) \right|, \qquad 0<x<1.
$$
As $\lim_{x \rightarrow 0}\, g(x,h)=1$ and for any $0<x\leq 1$ we have $g(x,h)>h$, for $0<h\leq 1$ (see Figure \ref{figS3}),  thus the condition \eqref{p3B} is satisfied.  Notice that $g(x,h)$ gets closer to the value $1$ as $h$ decreases. Therefore one can assure that realistic estimates \eqref{eq2A4} can be computed to approximate the true error of $S_3(f)$ for a step $h\leq 1$. In Table \ref{tabS3} is displayed the estimated errors $\bar E_{S_3} (f)$ and the true error for $S_3(f)$, respectively for $h\,\in\, \{  1/2,1/4,1/8,1/16\}$. As expected, the computed values for  $\bar E_{S_3} (f)$ have the correct sign and closely agree with the true error.

   \begin{figure}[h]
\begin{center} 
 \includegraphics[totalheight=5.2cm]{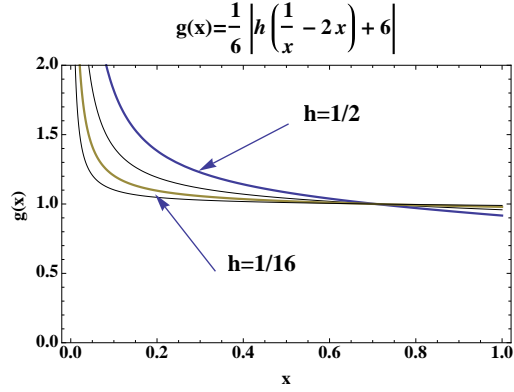}
\caption{\label{figS3} $g(x)>h$, for $h=1/2,1/4,1/8,1/16$.}
\end{center}
\end{figure} 
\noindent
  For $h=1/2$, we have
$$
\begin{array}{ll}
I(f)&= 0.746824 \\
Q_3(f)&= 2\, h\, f(0)=1\\
\tilde E_3(f)&= 2 h^2 f[0,1/2]+ \displaystyle{\frac{2\, h^3}{3}}  \, f[0,1/2,1]=     \\
& = -0.221199-0.03162046=-0.252850\\
S_3(f)&= Q_3(f)+\tilde E_3(f)=0.747180 .\\
\end{array}
$$
  
  \begin{table}[h!]
   $$
   \begin{array}{ | c |  c |  c|  }
   \hline
   h& \bar E_{S_3} (f)= \displaystyle{ \frac{-2\, h^3}{15} \frac{f[ x_1,x_2,x_3,\bar x_1, \bar x_2 ]}{f[x_1,x_2] } } \times \tilde E_3(f)& E_{S_3}(f)= I(f)-S_3(f)\\
   \hline
1/2&- 0.000396282&-0.000356296\\
   \hline
   1/4&- 0.000115228&-0.0000900798\\
   \hline
  1/8&- 4.92044\times 10^{-6}& -3.72994\times 10^{-6}\\
   \hline
1/16& -1.65494\times 10^{-7}& -1.24455\times 10^{-7}\\
   \hline
      \end{array}
   $$
   \caption{\label{tabS3}  Realistic and true error for $S_3 (f)$. }
\end{table}

\end{exam}

\begin{exam}\label{exemplo3modeloA}  (A realistic  error  for  the  $S_5(f) $ rule)

\bigbreak
\noindent
 From Table \ref{tabelamodeloA} we obtain the following expression for the rule $S_5(f)$,
  \begin{equation}\label{ruleS5}
  S_5(f)=4\, h\, f(x_1)+\left\{  8\, h^2 f[x_1,x_2] +\frac{40}{3}\, h^3  \, f[x_1,x_2,x_3]+ 16\, h^4\, f[x_1,\ldots,x_4] \right\}.
  \end{equation}
  Consider $I(f)=\int_0^{4 \,h} sin(2\,x) dx =sin^2(4\,h)$ and the interval $J=[a,b]=[0,\pi/5]$.  Since $f\in C^{6}(J)$ and $f'(x)=2\, \cos(2\,x)\neq 0,\,\, \forall \, x\in J$, the condition \eqref{cond1} holds with $x_1=0$ and $x_5=4\, h$. Let
$$
\begin{array}{ll}
g(x,h)&=  \left|   1+\displaystyle{ \frac{a_3\, f^{(2)} (x)}{a_2\, 2! \, f'(x)}     }     +\displaystyle{ \frac{a_4\, f^{(3)} (x)}{a_2\, 3! \, f'(x)}     }    +\displaystyle{ \frac{a_5\, f^{(4)} (x)}{a_2\, 4! \, f'(x)}     }         \right|, \qquad a< x <b     \\
&=  \left| 1 -\displaystyle{\frac{4\, h^2}{3}}+    \displaystyle{ \frac{1}{45} (14\, h^2-75)\, h\, \tan(2x)   }     \right|,
\end{array}
$$
where the coefficients $a_i$ are computed using \eqref{pesosparaA}.
It can be observed in the plot in the  Figure \ref{figS5} that for $h\in \{1/8,1/16,1/32,1/64\}$ the condition $g(x,h)>h$ is satisfied. Therefore,  since $n$ is odd,  one concludes from  Proposition \ref{prA} (a)  that the following realistic estimation for the true error of $S_5(f)$ is,
\begin{equation}\label{eqS5}
\begin{array}{ll}
\bar E_{S_5} (f)&= \displaystyle{  \frac{I(w_6)}{I(w_1)}  \frac{f[x_1,x_2,\ldots, x_5,\bar x_1, \bar x_2]}{f[ x_1,x_2 ]} }       \times \tilde E_5(f)\\
& = \displaystyle{  \frac{- 16\, h^5 }{21}    \frac{f[x_1,x_2,\ldots, x_5,\bar x_1, \bar x_2]}{f[ x_1,x_2 ]} }         \times \tilde E_5(f),
  \end{array}
\end{equation}
where $\bar x_1=(x_1+x_2)/2$ and $\bar x_2=(x_4+x_5)/2$. In Table \ref{tabS5} are displayed the computed realistic errors $\bar E_{S_5}(f)$ for the steps $h$ referred above.

   \begin{figure}[h]
\begin{center} 
 \includegraphics[totalheight=5.2cm]{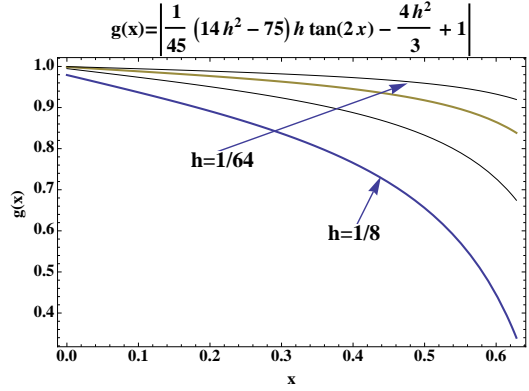}
\caption{\label{figS5} $g(x,h)>h$, for $h=1/8,1/16,1/32,1/64$.}
\end{center}
\end{figure} 
\noindent
  For instance for $h=1/8$, we have
$$
\hspace{-1cm}\begin{array}{ll}
I(f)&= I=0.2298488470659301 \\
Q_5(f)&= 4\, h\, f(x_1)=0\\
\\
\tilde E_5(f)&= 8\,  h^2 f[x_1,x_2]+\frac{40}{3} \, h^3 f[x_1,x_2,x_3]+ 16 h^4 f[x_1,\ldots,x_4]+\frac{112}{15} h^5 f[x_1,\ldots,x_5]=     \\
& = 0.2474039592545229  -0.01281864992070238   -\\
&-0.004808659341902015 +   0.00007207430695446034 \\
&= 0.2298487242988730,
\end{array}
$$

\noindent
and finally
$$
S_5(f)= Q_5(f)+\tilde E_5(f)=0.229848724298873. 
$$
  
  \begin{table}[h!]
   $$
   \begin{array}{ | c |  c |  c|  }
   \hline
   h& \bar E_{S_5} (f)\simeq \displaystyle{ \frac{-16 \, h^5}{21} \frac{f[ x_1,x_2,x_3,x_4,x_5,\bar x_1, \bar x_2 ]}{f[x_1,x_2] } } \times \tilde E_5(f)& E_{S_5}(f)= I(f)-S_5f)\\
   \hline
1/8&1.14143\times 10^{-7}& 1.22767\times 10^{-7}\\
   \hline
   1/16& 4.89318\times 10^{-10}& 4.98246\times 10^{-10}\\
   \hline
  1/32& 1.95599\times 10^{-12}& 1.96484\times 10^{-12}  \\
   \hline
1/64& 7.68478\times 10^{-15} & 7.69335\times 10^{-15}\\
   \hline
      \end{array}
   $$
   \caption{\label{tabS5}  Realistic and true error for $S_5 (f)$. }
\end{table}
\end{exam}

\section{Composite  rules}\label{seccompositeA}
The  rules whose  weights  have been  given in  Table \ref{tabelamodeloA} are here applied  in order to obtain the so\--called {\it composite} rules.  The algorithm described hereafter for composite rules is illustrated by several examples  presented in paragraph \ref{subsecexemplosA}. Since the best rules $S_n$ are the ones for which $n$ is even (when   $deg(S_n)=n$ holds)  the examples refer to $S_3$, $S_5$,  $S_7$ and $S_9$. Whenever the conditions of Proposition \ref{prA} for obtaining realistic errors are satisfied, these rules enable the computation of high precision approximations to the integral $I(f)$,   as well  as good approximations to the true error. This justifies the name {\it realistic} error adopted in this work. 

\bigbreak
\noindent
Let $n\geq 2$ be given and fix a natural number $i$. Consider the number $N=(n-1)\times i$ and divide the interval $[a,b]$ into $N$ equal parts of length $h=(b-a)/N$, denoting by  $x_1,x_2,\ldots, x_N$ the nodes, with $x_1=a$, $x_N=b$, and  $x_j=x_{j-1}+h,\quad j=2,3,\ldots (N-1)$. Partitioning the set $\{x_1,x_2,\ldots, x_N\}$ into subsets of $n$ points each,  and  for an offset of $n-1$ points, we get $i$ panels each one containing  $n$ successive nodes. To each panel we apply in succession the rule $Q_n$ and compute the respective realistic correction $\tilde E_n$ as well as the estimated realistic error $\bar E_{S_n}$ for the rule $S_n(f)$.  For the output we compute the sum of the partial results obtained for each panel as described in \eqref{panels}: 
\begin{equation}\label{panels}
\left\{
\begin{array}{l}
Q=\sum_{k=1}^{i} Q_n(\mbox{panel}_k)\\
\tilde E= \sum_{k=1}^{i} \tilde E_n(\mbox{panel}_k)\\
S=Q+\tilde E\quad \mbox{(composite rule)}\\
\bar E= \sum_{k=1}^{i} \bar E_n(\mbox{panel}_k)\quad \mbox{(realistic error for $S$)}.
\end{array}
\right.
\end{equation}
According to Proposition \ref{pCA}, 
 for composite rules with $n$ points by panel and step $h>0$, in the favorable cases (those satisfying the hypotheses  behind  the theory) one can expect to be able to compute approximations $S$ having a realistic error. Analytic proofs for  realistic errors in composite rules for both models $A$ and $B$ will be treated in a forthcoming work \cite{mg1}.
\subsection{Numerical examples for composite rules}\label{subsecexemplosA}

\noindent
Once computed realistic errors within each panel for a composite rule, we can expect the error $\bar E$ in \eqref{panels} to be  also realistic. This happens in all the numerical examples worked below.
%
\begin{exam}\label{exemplo1A}
\end{exam} 

\noindent
Let 
$$
\begin{array}{ll}
I(f)&=\int_{10^5}^{2\times 10^5}\displaystyle{ \frac{1}{\ln(x)}    } dx\qquad \mbox{(\cite{steffensen}, p. 161   )}\\
&= li(2\times 10^5)-li(10^5) 
\end{array}
$$
In the interval $J=[10^5, 2\times 10^5]$,  the function $f(x)=1/\ln(x)$ belongs to the class $C^{\infty} (J)$. Since $f'(x)= -( x\, \ln^2(x))^{-1}\neq 0$, for all $x\in J$, and for $k\geq 2$ the quotients $f^{(k)} (x)  / f'(x)$,  with $x\in J$, are close to $\gamma (x)=0$ and tends to the zero function  $\gamma(x)$ as $k$ increases. Therefore, the lefthand side in the inequality   \eqref{cond2} is very close to $1$. That is, for $n\geq 2$ the function
$$
g(x, h)= \left|1+    \sum_{j=2}^{n-1} \displaystyle{ \frac{a_{j+1}}{a_2}  }    \,   \displaystyle{  \frac{f^{(j)}(x)}{j!\, f'(x)}} \right| 
$$
is such that $g(x,h)\simeq 1$, for $h$ sufficiently small. So Proposition \ref{prA} holds and  one obtains  realistic errors  for the rules $S_n(f)$.

\noindent
The behavior of the function $g(x,h)$ is illustrated for  the case $n=3$ in Figure \ref{figS3logint},   for $h\,\in \{5, 5/2, 5/3   \}$. Note that in this example $g(x)\simeq 1$ while the chosen steps $h$ are greater than $1$. However, realistic errors are still obtained.

   \begin{figure}[h]
\begin{center} 
 \includegraphics[totalheight=5.2cm]{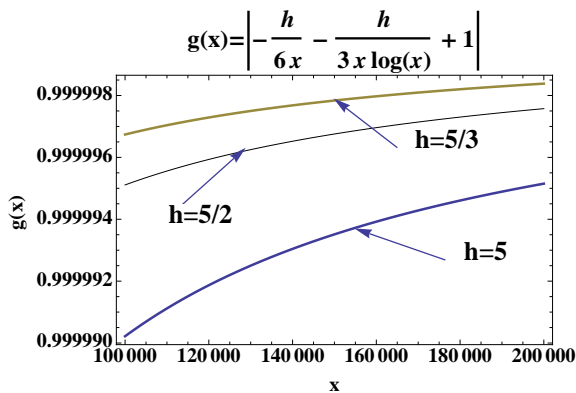}
\caption{\label{figS3logint}  $n=3$.  $g(x)\simeq 1$, for $h=5,5/2,5/3$.}
\end{center}
\end{figure} 

\noindent
 Using a precison of  $32$  decimal digits (or greater for $n>3$) the following values are obtained for the $3$\--point composite rule:
 $$
 \begin{array}{ll}
I&=8406.2431208462027086216460436947\\
h&=5\\
Q&=8406.2677835091928175\\
E&= -0.024662837510842808609  +1.7452073401705812569\times 10^{-7}=  \\
&=      -0.024662662990108791550\\
S&=8406.2431208462027087\\
\bar E_S&=-5.9854000\times 10^{-17}   \quad \mbox{ (realistic error for $S$)   }\\
I-S&=-5.9854472*10^{-17}\quad \mbox{(true error)}.
 \end{array}
 $$
 \noindent
In Table \ref{tabexemplo3b} the realistic error is compared with  the true error, respectively for the rules with $n$ odd from  $3$ to $9$ points  (the step $h$  is as tabulated).

\begin{table}[h]
\hspace{-1cm}$$\begin{array}{|c|c|c| c|}\hline
n&h&\bar E_S&I-S\\
\hline
3&5&-5.98540\times 10^{-17}&-5.98545\times 10^{-17}\\
\hline
3&5/3&-7.38942\times 10^{-19}&-7.38944\times 10^{-19}\\
\hline
5&5/2&-1.30573\times 10^{-26}&-1.30576\times 10^{-26}\\
\hline
5&5/6&-1.79116\times 10^{-29}&-1.79117\times 10^{-29}\\
\hline
7&5/3&-5.31897\times 10^{-36}&-5.31911\times 10^{-36}\\
\hline
7&5/6&-2.07775\times 10^{-38}&-2.07778\times 10^{-38}\\
\hline
9&25/6&-4.95560\times 10^{-40}&-4.95608\times 10^{-40}\\
\hline
9&5/2&-2.99658\times 10^{-42}&-2.99675\times 10^{-42}\\
\hline
  \end{array}$$
\caption{\label{tabexemplo3b}  Comparison of the realistic error $\bar E_S$ with the true error. }
\end{table}

\noindent
For $n=7$, the computed approximation for the integral $I$ is
$$
S=8406.24312084620270862164604369467068,
$$
where all the digits are correct. The simple  rule $S_7(f)$ is defined (see Table \ref{tabelamodeloA}) as
\begin{equation}\label{ruleS7}
\begin{array}{ll}
S_7(f)&= 6\, h\,f(x_1)+\left\{  18 h^2 f[x_1,x_2]+ 54 h^3 f[x_1, x_2,x_3]+144 h^4 f[x_1,\ldots,x_4]+ \right.\\
&+ \left.  \frac{1476}{5} h^5 f[x_1,\ldots,x_5]+ 396 h^6 f[x_1,\ldots, x_6]+\frac{1476}{5} f[x_1,\ldots, x_7]    \right\}.
\end{array}
\end{equation}
The respective realistic error is (see \eqref{pnA})
\begin{equation}\label{realS7}
\bar E_{S_7}= \frac{I(w_8)}{i(w_1)}\frac{d_8}{d_1}\times \tilde E_7= - \frac{72}{5}\, h^7   \frac{d_8}{d_1}\times \tilde E_7.
\end{equation}
In Appendix \ref{apendiceA} a  {\sl Mathematica} code for the composite rule $S$, for $n=7$,  is given. The respective procedure is called $Q7A$ and the code includes comments explaining the respective algorithm. Of course we could have adopted a more efficient  programming style, but our goal here is simply to illustrate the algorithm described above for the composite rules.


%
%
 
\begin{apen}\label{apendiceA}   (Composite rule $S$ for $n=7$ points )

$$
\begin{array}{l}
\mbox{(*   For the data } \,\, \{ x_1,\ldots, x_n   \}, \,  \{ f_1,\ldots, f_n   \},\,\, \mbox{and}\,\, k\geq0 \\
\mbox{ the output  for $d[x,y,k,p]$ is the divide difference}\, f[x_1,\ldots, x_k]  \,\, \mbox{*)}\\
\mbox{(* The user may enter a precision $p$ which will be assigned to the data.   *)}\\
\mbox{(*  By default $p=\infty$  *)}\\
\end{array}
$$
$$
\begin{array}{l}
d[xi\verb+_+List, yi\verb+_+List, k\verb+_+ /; k >= 0, precision\verb+_+: Infinity] :=\\
  
  Module[\{n = Length[xi],  x, y, dd\},\\
 \hspace{3cm}   \mbox{(*  set default precision to nodes xi  *)}\\
   x = SetPrecision[xi, precision];  \\
   \hspace{3cm}   \mbox{ (*    set default precision to functional values yi *)}\\
   y = SetPrecision[yi, precision];  \\
 \mbox{  (*   recursive definiton for finite differences :   *)  }\\
   \end{array}
   $$
   $$
   \begin{array}{l}
   dd[0, j\verb+_+] := y[[j]];   \mbox{ (* order 0 difference  *) }\\
 \mbox{    (*   ordem i difference; dynamic computation  *)  }\\
   
   dd[i\verb+_+, j\verb+_+] :=   dd[i, j] = (dd[i - 1, j + 1] - dd[i - 1, j])/(x[[i + j]] - x[[j]]);\\
  \mbox{  (* Output:  first difference of order k *)  } \\
   dd[k, 1]\quad  ];\\
   \end{array}
$$

$$
\begin{array}{l}
\mbox{(*   The procedure $q7A$ uses the algorithm for the simple rule with  n=7} \\
\hspace{5cm} \mbox{ nodes  *) }     \\
\mbox{(* The output is a list containing the relevant items for this simple rule *)}\\
\\
q7A[\{\{t1\verb+_+, f1\verb+_+\}, \{t2\verb+_+, f2\verb+_+\}, \{t3\verb+_+, f3\verb+_+\}, \{t4\verb+_+, f4\verb+_+\}, \{t5\verb+_+, f5\verb+_+\}, \{t6\verb+_+,  f6\verb+_+\},\\
\hspace{5cm}  \{t7\verb+_+, f7\verb+_+\}\}, precision\verb+_+: Infinity] := \\
 Module[\{ x1, x2, x3, x4, x5, x6 , x7, xb1, xb2, y1, y2, y3, y4, y5,   y6, y7,   \\
 \hspace{3cm} yext, hh, ext, d1, d8, q, e1, e2, e3, e4, e5, e6, E7, s,  real\},\\
 \end{array}
 $$
 $$
 \begin{array}{l}
  \{y1, y2, y3, y4, y5, y6, y7\} =  SetPrecision[\{f1, f2, f3, f4, f5, f6, f7 \}, \\
  \hspace{5cm} precision];\\
  \{x1, x2, x3, x4, x5, x6, x7\} =   SetPrecision[\{t1, t2, t3, t4, t5, t6, t7 \},\\
  \hspace{5cm} precision] ;\\
  xb1 = (x1 + x2)/2; xb2 = (x6 + x7)/2;\\
  ext = \{x1, x2, x3, x4, x5, x6, x7, xb1, xb2\}; \\
  yext = Map[f, ext]; \quad \mbox{(*   completation of the panel *) }\\
  d1 = d[\{x1, x2\}, \{y1, y2\}, 1];  \quad \mbox{(*  divided difference order 1      *)}\\
  \end{array}
  $$
  $$
  \begin{array}{l}
  d8 = d[ext, yext, 8, precision]; \quad \mbox{(*  divided difference order 8      *)}\\
    hh = SetPrecision[h, precision]; \quad  \mbox{ (* step  *)} \\
  q = 6*  hh*y1; \quad \mbox{ (*  left rectangle quadrature  rule *)  }\\
  e1 = 18* hh^2*d[\{x1, x2\}, \{y1, y2\}, 1]; \quad \mbox{  (* error e1 *)}\\
  e2 = 54* hh^3 *d[\{x1, x2, x3\}, \{y1, y2, y3\}, 2]; \quad \mbox{ (* error e2 *) }\\
  \hspace{5cm}  \mbox{ (* error e3  :   *) }\\
  e3 = 144* hh^4*d[\{x1, x2, x3, x4\}, \{y1, y2, y3, y4\}, 3];  \\
   \hspace{5cm}  \mbox{ (* error e4  :   *) }\\
   e4 = 1476/5* hh^5*   d[\{x1, x2, x3, x4, x5\}, \{y1, y2, y3, y4, y5\}, 4];  \\
       \hspace{5cm}  \mbox{ (* error e5  :   *) }\\
 e5 = 396* hh^6* d[\{x1, x2, x3, x4, x5, x6\}, \{y1, y2, y3, y4, y5, y6\}, 5]; \\
  \hspace{5cm}  \mbox{ (* error e6  :   *) }\\
 e6 = 1476* hh^7/7*  d[\{x1, x2, x3, x4, x5, x6, x7\},\\
 \hspace{5cm}  \{y1, y2, y3, y4, y5, y6, y7\},  6]; \\
  E7 = e1 + e2 + e3 + e4 + e5 + e6; \quad \mbox{  (*   realistic error for q  *)}\\
  s = q + E7;\quad \mbox{ (* s is a realistic approximation to the integral  *) }\\
  \end{array}
  $$
  $$
  \begin{array}{l}
  real = N[-72* hh^7/5 * d8/d1*E7,  8]; \quad \mbox{ (* realistic error for s  *) }\\
  \{q, e1, e2, e3, e4, e5, e6, E7, s, real\}]; \quad \mbox{ (* output  *)}
  \end{array}
  $$
  $$
  \begin{array}{l}
  \mbox{(*   The following procedure $Q7A$ gives a list containing the relevant }\\ 
  \hspace{2cm} \mbox{ items for the composite rule. It  calls the procedure $q7A$:     *)}  \\
  \\
  Q7A[x\verb+_+List, y\verb+_+List, precision\verb+_+: Infinity] := Module[\{data, list\},\\
  \hspace{2cm}  \mbox{ (*  partition into 7 nodes offset 6  *) }\\
   data = Partition[Transpose[\{x, y\}], 7, 6] ; \\
   \hspace{2.5cm} \quad \mbox{ (* sum entries in cells and prepend the step  h used  *)}\\
 list = Map[q7A[\verb+#+, precision] \verb+&+, data]; \\
    Prepend[  Map[Apply[Plus,\verb+ #+] \verb+&+, Transpose[list]], Rationalize[h]] \,\,  ];\\
  \end{array}
  $$
 
\end{apen}

\noindent
{\bf Acknowledgments}

\noindent
This work  has been supported by  Instituto de Mecânica\--IDMEC/IST, Centro de Projecto Mecânico, through FCT (Portugal)/program POCTI.

\noindent



\end{document}